\documentclass[preprint,12pt]{elsarticle}

\usepackage{amsfonts,amsthm,amsmath}
\usepackage[dvips]{epsfig}
\usepackage[dvips]{graphicx}

\theoremstyle{plain}
\numberwithin{equation}{section}
\newtheorem{theorem}[equation]{Theorem}

\newtheorem{proposition}[equation]{Proposition}
\newtheorem{example}[equation]{Example}
\newtheorem{corollary}[equation]{Corollary}
\newtheorem{remark}[equation]{Remark}
\newtheorem{openprob}[equation]{Open problem}
\newtheorem{definition}[equation]{Definition}

\begin{document}

\begin{frontmatter}

\title{Extensions of rich words}

\author{Jetro Vesti}
\ead{jejove@utu.fi}


\address{University of Turku, Department of Mathematics and Statistics, 20014 Turku, Finland}

\begin{abstract}  
In \cite{djp}, it was proved that every word $w$ has at most $|w|+1$
many distinct palindromic factors, including the empty word.
The unified study of words which achieve this limit was initiated in \cite{gjwz}.
They called these words \emph{rich} (in palindromes).

This article contains several results about rich words and especially extending them.
We say that a rich word $w$ can be \emph{extended richly} with a word $u$ if $wu$ is rich.
Some notions are also made about the infinite defect of a word, the number of rich words of length $n$
and two-dimensional rich words.
\end{abstract}

\begin{keyword} Combinatorics on words, Palindromes, Rich words, Sturmian words, Defect, Two-dimensional words.
\MSC 68R15
\end{keyword}

\journal{x}

\end{frontmatter}


\section{Introduction}

In \cite{djp}, it was proved that every word $w$ has at most $|w|+1$ many distinct palindromic factors, including the empty word.
The authors of \cite{gjwz} recently initiated a unified study of words which achieve this limit. They called these words \emph{rich} (in palindromes).
The same concept was introduced already in \cite{bhnr} with the term \emph{full} words. Full, meaning that a bound is reached.
This class of words have been studied in several other papers from various point of views,
for example in \cite{afmp}, \cite{bdgz1}, \cite{bdgz}, \cite{lgz} and \cite{rr}.

In Section 2 we will prove several results about extending rich words.
A rich word $w$ can be \emph{extended richly} with a word $u$ if $wu$ is rich.
In \cite{gjwz} it was proved that every rich word can be extended at least with one letter.
We will prove that every rich word $w$ can be eventually extended richly at least in two ways,
after it had been extended at most with a word of length $2|w|$. This fact will be used in
several places later. We will also show that every rich word can be extended to be both
an infinite aperiodic and infinite periodic rich word. Also, all Sturmian words can be extended richly in two ways.

In Section 3 we will define a new concept, the \emph{infinite defect} of a finite word.
The normal \emph{defect} of a word $w$ tells how many palindromes does it lack.
The infinite defect tells how many palindromes does the infinite extension of $w$
has to lack in minimum. We will notice that this number is always finite.
We will also give some upper bounds and study how the defect and the infinite defect
can differ from each other. 

In Section 4 we will give upper and lower bounds for the number of rich words of length $n$.
Clearly all words are not rich, so it is natural to study how many of them exist. This has not yet been done.
The upper bound has no theory behind it, but it seems to be quite near to the actual number of rich words.
For the lower bound we will use the fact that every rich word could be extended at least in two different ways after a limited extension.

In Section 5 we will shortly study two-dimensional rich words and their extensions.

In Section 6 we will give some open problems from the previous sections.

\subsection{Definitions and notation}

An \emph{alphabet} $A$ is a non-empty finite set of symbols, called \emph{letters}.
A \emph{word} is a finite sequence of letters from $A$.
The \emph{empty} word $\epsilon$ is the empty sequence.
The set $A^*$ of all finite words over $A$ is a \emph{free monoid} under the operation of concatenation.
The \emph{free semigroup} $A^+=A^*\setminus \{\epsilon\}$ is the set of non-empty words over $A$.

A \emph{right infinite word} (resp. \emph{left}) is a sequence indexed by $\mathbb{Z}_+$
(resp. $\mathbb{Z}_-$) with values in $A$.
A \emph{two-way infinite word} is a sequence indexed by $\mathbb{Z}$.
We denote the set of all infinite words over $A$ by $A^\omega$ and define $A^\infty=A^*\cup A^\omega$.
An infinite word is \emph{ultimately periodic} if it can be written as $uv^\infty=uvvv\cdots,$
for some $u,v\in A^*,$ $v\neq\epsilon$. If $u=\epsilon,$ then we say the infinite word is \emph{periodic}.
An infinite word that is not ultimately periodic is \emph{aperiodic}.

The \emph{length} of a word $w=a_1 a_2 \ldots a_n \in A^+$, with each $a_i\in A$, is denoted by $|w|=n$.
The empty word $\epsilon$ is the unique word of length $0$. By $|w|_a$ we denote the number of occurences of a letter $a$ in $w$. 
The \emph{reversal} of $w$ is denoted by $\bar{w}=a_n\ldots a_2 a_1$.
Word $w$ is called a \emph{palindrome} if $w=\bar{w}$.
The empty word $\epsilon$ is assumed to be a palindrome.

A word $x$ is a \emph{factor} of a word $w\in A^\infty$ if $w=uxv$, for some $u,v\in A^\infty$.
If $u=\epsilon$ ($v=\epsilon$) then we say that $x$ is a \emph{prefix} (resp. \emph{suffix}) of $w$.
A factor $x$ of a word $w$ is said to be \emph{unioccurent} in $w$ if $x$ has exactly one occurence in $w$.
Two occurences of factor $x$ are said to be \emph{consecutive} if there is no occurence of $x$ between them.
A factor of $w$ having exactly two occurences of a non-empty factor $u$, the other as a prefix and the other as a suffix,
is called a \emph{complete return} to $u$ in $w$.

If $w=uv\in A^+$, by notation $u^{-1}w=v$ or $wv^{-1}=u$ we mean the removal of a prefix or a suffix of $w$.
The \emph{right palindromic closure} (resp. \emph{left}) of a word $w$ is the unique shortest palindrome
$w^{(+)}$ (resp. $^{(+)}w$) having $w$ as a prefix (resp. suffix). If $u$ is the (unique) longest palindromic
suffix of $w=vu$ then $w^{(+)}=vu\bar{v}$.

Let $w$ be a finite or infinite word. The set $\textrm{F}(w)$ means the set of all factors of $w$,
the set $\textrm{Alph}(w)$ means the set of all letters that occur in $w$ and
the set $\textrm{Pal}(w)$ means the set of all palindromic factors of $w$.
We say that a word $w$ is \emph{unary} if $|\textrm{Alph}(w)|=1$, \emph{binary} if $|\textrm{Alph}(w)|=2$,
\emph{ternary} if $|\textrm{Alph}(w)|=3$ and \emph{$n$-ary} if $|\textrm{Alph}(w)|=n$

Other basic definitions and notation in combinatorics on words can be found from Lothaires books \cite{l1} and \cite{l2}. 

\subsection{Basic properties of rich words}

In this subsection we give the basic definitions and state some already known
properties and characterizations of rich words.

\begin{proposition}\label{p0}(\cite{djp}, Prop. 2)
A word $w$ has at most $|w|+1$ distinct palindromic factors.
\end{proposition}

\begin{definition}
A word $w$ is \emph{rich} if it has exactly $|w|+1$ distinct palindromic factors.
\end{definition}

\begin{definition}
An infinite word is \emph{rich} if all of its factors are rich.
\end{definition}

\begin{proposition}\label{p1}(\cite{gjwz}, Cor. 2.5)
A word $w$ is rich if and only if all of its factors are rich.
\end{proposition}

\begin{proposition}\label{p2}(\cite{gjwz}, Cor. 2.5)
If $w$ is rich, then it has exactly one unioccurent longest palindromic suffix (referred later as  ${\rm lps}$ or ${\rm lps}(w)$).
\end{proposition}

From Cor. 2.5 in \cite{gjwz} we also get that if $w$ is rich then $\bar{w}$ is rich.
From this we see that the above proposition holds for prefixes also and we refer to
the unioccurent longest palindromic prefix of $w$ as $\textrm{lpp}$ or $\textrm{lpp}(w)$.

The next proposition characterizes rich words according to complete return words.

\begin{proposition}\label{p3}(\cite{gjwz}, Thm. 2.14)
A finite or infinite word $w$ is rich if and only if
all complete returns to any palindromic factor in $w$ are themselves palindromes.
\end{proposition}

We can also take the palindromic closure of a rich word and get a rich word.

\begin{proposition}\label{p4}(\cite{gjwz}, Prop. 2.6)
Palindromic closure preserves richness. 
\end{proposition}

Let $w$ be a word and $u\neq w$ its longest proper palindromic suffix.
The \emph{proper palindromic closure} of $w=vu$ is defined as $w^{(++)}=vu\bar{v}$.
From the proof of Proposition 2.8 in \cite{gjwz} we get that also the proper palindromic
closure preserves richness. It uses the fact that the longest proper palindromic suffix
(referred later as $\textrm{lpps}(w)$ or $\textrm{lpps}$) can occur only in the beginning
and in the end of the word. Hence we create a new palindrome in every step when we are
taking the proper palindromic closure and the word stays rich.

\begin{proposition}\label{p5}(\cite{gjwz}, proof of Prop. 2.8)
The proper palindromic closure preserves richness.
\end{proposition}


\section{Extensions of rich words}

Let $w$ be a rich word. We say that $w$ can be \emph{extended richly} with a word $u$ if $wu$ is rich.
The word $wu$ is called a \emph{rich extension} of $w$ with the word $u$.
We also say that $w$ can be extended richly in $n$ ways if there exists $n$ distinct letters $a\in\textrm{Alph}(w)$ such that $wa$ is rich.
The word $w$ can be \emph{eventually extended richly} in $n$ ways if there exists a finite word $u$ such that
the word $wu$ is rich and can be extended richly in $n$ ways.

In \cite{gjwz}, it was proved that every rich word can be extended richly at least in one way:

\begin{proposition}\label{p6}(\cite{gjwz}, Prop. 2.8)
Suppose $w$ is a rich word. Then there exist letters $x,z\in{\rm Alph}(w)$ such that $wx$ and $zw$ are rich.
\end{proposition}

The next theorem states that every (non-unary) rich word can be eventually extended richly at least in two ways.

\begin{theorem}\label{tt1}
Let $w$ be a non-unary rich word.
There exists a word $u$ such that $wu$ is rich, $|u|<2|w|$ and $wu$ can be extended richly at least in two ways.
\end{theorem}
\begin{proof}
The idea of the proof is to take the biggest power $a^n$ of some letter $a$ in $w$ and
then extend the word richly such that this biggest power is the suffix of $wu$
(we stop extending when we see it as a suffix for the first time).
Now we can extend the word richly with the letter $a$ and with a letter $b$ that is
before the lpps$(wu)$, where $b$ is a different letter than $a$ because $a^n$ was the biggest power of $a$.
The extension with $a$ gives a new palindrome $a^{n+1}$ and the extension with $b$ gives a new palindrome $b$lpps$b$.

Now we only have to show that we actually can extend $w$ richly such that $a^n$ is a suffix.
Suppose $w=zv$, where $v$ is the longest proper palindromic suffix. Now $w^{(++)}=zv\bar{z}$.
If $a^n$ is a factor of $z$ then we can cut the word $w^{(++)}$ such that the $a^n$ inside $\bar{z}$ is a suffix.
If $a^n$ is a prefix, and hence a suffix, of $v$ then we are already done.
So we still need to consider the case where $a^n$ is as a factor only inside $v$.

For this case we take the proper palindromic closure of $w^{(++)}$.
The word $x=\textrm{lpps}(w^{(++)})$ cannot be longer than $v\bar{z}$ because otherwise $v$ would not be lpps$(w)$.
Now there are two subcases:
1) $v\bar{z}=v'x$, where $v'$ is a prefix of $v=v'v''$,
and
2) $\bar{z}=z'x$.

Suppose the case 1. Now the word $zv\bar{z}\bar{v'}$ is rich and if $a^n$ is a factor of $v'$
then we can cut the word $zv\bar{z}\bar{v'}$ inside $\bar{v'}$ such that $a^n$ is a suffix.
If $a^n$ is not a factor of $v'$ then it has to be a factor of $v''$, which is a contradiction.
This comes from the fact that $v'v''\bar{z}$ would be an overlap of two palindromes $v$ and $x$,
where $v''$ is the common part which contains $a^n$ but the other parts $v'$ and $\bar{z}$ would not contain $a^n$.

Suppose the case 2. Now the word $zvz'x\bar{z'}v$ is rich and because $v$ contains $a^n$ as a factor
we can cut the word $zvz'x\bar{z'}v$ inside the latter $v$ such that $a^n$ is a suffix.

The rich extension $wu$ was clearly constructed in every case such that $|u|<2|w|$.
\end{proof}

\begin{example}
Let us use the idea from the previous theorem into the following rich word
$w=11011010101010110011000111000011100$. The word can be extended richly only with letter $0$
and the biggest powers of both letters $0$ and $1$ are inside the lpps $00111000011100$.
We take the proper palindromic closure $w^{(++)}=w011001101010101011011=wz'x$,
where the new lpps is $x=11011$. Now we get $111$ to a suffix: $(w^{(++)})^{(++)}=wz'x\bar{z'}00111\cdots$
and we can extend the word richly with both $0$ and $1$.
\end{example}

\begin{remark}
The original word in the previous example had length $35$ and the new word
up to the point where we can extend it with two ways had length $77$.
We can increase the ratio $77/35=2.2$ ultimately close to $3$ by making the block $010101010$ longer.
So the bound from the previous theorem can be reached if we use the idea of the proof on how to extend the word richly.
Although, already the word $w0$ can be extended richly in two ways! This implies that there is much to improve.
\end{remark}

\begin{remark}\label{rr1}
We can construct words such that the number of consecutive compulsory rich extensions grows arbitrarily large.
For example for the word $01011011101111011111001$ there are four compulsory extensions
with letter $1$ before we can extend it with both $0$ and $1$.
From the above word we can easily see how to generalize this for an arbitrary $n$.
Notice that the length of the word is $\frac{(n+1)n}{2}+3$ for $n$ compulsory rich extensions, so the word grows rapidly.
There can also be consecutive compulsory rich extensions which are not the for the same letter.
For example the word $1010010011000110010$ has to be extended first with letter $0$ and then with letter $1$.
\end{remark}

The next proposition will be used in Proposition \ref{pp1} and Proposition \ref{pp2}.
It gives us a necessary condition if two rich words can appear in a same rich word.

\begin{proposition}\label{pj}
Two rich words $u$ and $v$ cannot be factors of a same rich word if there are
$u'\in{\rm F}(u)$ and $v'\in{\rm F}(v)$ such that
${\rm lps}(u')={\rm lps}(v'),\ {\rm lpp}(u')={\rm lpp}(v')$ and $u'\neq v'$.
\end{proposition}
\begin{proof}
From Theorem 6 in \cite{bdgz} we get immediately that if some word would contain such words $u$ and $v$ as factors,
and hence $u'$ and $v'$, then it would not be rich.
\end{proof}

The next proposition states that every rich word cannot be eventually extended richly in $3$ ways.
This means Theorem \ref{tt1} is not true for extensions in three or more ways.

\begin{proposition}\label{pp1}
Every ternary rich word cannot be eventually extended richly in $3$ ways.
\end{proposition}
\begin{proof}
The word $w=0020102202$ is rich. Let us prove that it cannot be eventually extended richly in 3 ways.
First we give some forbidden factors that can never appear on any rich extension of $w$.
We use Proposition \ref{pj} for this.

The factors $12$, $21$, $001$ and $0202$ are forbidden immediately due to Proposition \ref{pj}
because $w$ has factors $102$, $201$, $00201$ and $020102202$, respectively.
Suppose some rich extension of $w$ would have factor $00$. Then we could take the first occurence of $00$
and get that the rich extension has factor $\bar{w}=2022010200$. This is because the complete return to $00$ has to be a palindrome.
Now we would have factors $22010200$ and $10200$, which means that the factors $2200$ and $100$ are also forbidden

Now suppose the contrary: there exists a rich extension $wu$ such that $wu0$, $wu1$ and $wu2$ are rich.
The last letter of $wu$ has to be $0$ because otherwise we would have factors $12$ or $21$.
Now we have three cases depending on which is the second last letter of $wu$.

1) Suppose $wu=x00$. This  would give a forbidden factor $001$ in $wu1=x001$.

2) Suppose $wu=x10$. This  would give a forbidden factor $100$ in $wu0=x100$.

3) Suppose $wu=x20$.
If the third last letter is $0$ then we would get a factor $0202$ in $wu2$.
If the third last letter is $1$ then we would get a factor $12$ already in $wu$.
If the third last letter is $2$ then we would get a factor $2200$ in $wu0$.
These are all forbidden factors.

In each case we get a contradiction and the proof is complete.
\end{proof}

The next proposition states the previous for every $n\geq3$.

\begin{proposition}\label{pp2}
For every $n\geq3$ there exists an $n$-ary rich word which cannot be eventually extended richly in $n$ ways.
\end{proposition}
\begin{proof}
The case $n=3$ follows from the previous proposition. For $n\geq4$ we take the word $w=123\cdots(n-1)n$.

Similar to the previous proof, factors $13,31,24,42,\ldots,(n-2)n$ and $n(n-2)$ can never appear on any rich extension of $w$.
Suppose to the contrary that $wu$ is such that we can extend it richly in $n$ ways, i.e. with all the letters $1,2,\ldots,n$.
Now, for every last letter of $wu$ we would always get one of the forbidden factors listed above.
\end{proof}

\begin{remark}
For every $n\geq1$ there also exists $n$-ary rich word which can be extended richly in $n$ ways.
The word $50102010301020104010201030102010$ is an example for $n=6$.
When you see how the above word is constructed you see that it can be generalized for every $n\geq1$.
\end{remark}

\begin{remark}
From the proof of the previous proposition we also see that if a word have factors like $0^k 1^l 2^m 3^n$,
where $k,l,m,n\geq1$, then we can never extend the word richly with all the letters $0,1,2$ and $3$.
\end{remark}

Using Theorem \ref{tt1} and its idea we can extend every rich word to be an infinite aperiodic rich word.

\begin{proposition}\label{pp3}
If $w$ is a non-unary rich word then it can be extended to be an infinite aperiodic rich word.
\end{proposition}
\begin{proof}
We can construct such word for every $w$ by repeating the procedure of Theorem \ref{tt1} infinitely many times.
We first choose the biggest power $a^n$ of some letter $a$ in $w$. Then we extend the word such that
this biggest power is a suffix. Then we extend the word with the letter $a$ and we get a new palindrome $a^{n+1}$.
Then we extend the word with the letter that is before the factor $a^{n+1}$. When we repeat this procedure
we get longer and longer powers of some letter. This kind of word is clearly aperiodic and rich.
\end{proof}

Using Proposition \ref{p4} we can also extend every rich word to be an infinite periodic rich word.

\begin{proposition}\label{pp4}
If $w$ is rich then it can be extended to be an infinite periodic rich word. 
\end{proposition}
\begin{proof}
From Proposition \ref{p5} we get that $w^{(++)}=uv$ is a rich palindrome, where $v$ is the lpps of $w^{(++)}$.
We prove by induction that $w^{(++)^n}=u^n v$, where $w^{(++)^n}$ means taking the proper palindromic closure
$n$ times in a row. It holds for $n=1$. Suppose it holds for $n=k$.

Now we have to prove that $w^{(++)^{k+1}}=u^{k+1} v$. When we use the assumption that the claim was true for $n=k$,
we get $w^{(++)^{k+1}}=(u^k v)^{(++)}=(v (\bar{u})^k)^{(++)}=(\bar{u}v(\bar{u})^{k-1})^{(++)}$.
Because $v(\bar{u})^{k-1}=u^{k-1}v$ is a palindrome, we get that ${\rm lpps}(\bar{u}v(\bar{u})^{k-1})=v(\bar{u})^{k-1}$.
Otherwise $v(\bar{u})^{k-1}$ would occur somewhere else than in the end or in the beginning of the word $\bar{u}v(\bar{u})^{k-1}$
and hence $v$ would not be the lpps of $w^{(++)}$.
This all means that $w^{(++)^{k+1}}=(\bar{u}v(\bar{u})^{k-1})^{(++)}=\bar{u}v(\bar{u})^{k-1}\bar{u}=u^{k+1}v$.

Repeating the procedure infinitely many times we get an infinite periodic word $u^\infty$ which is rich and has $w$ as a prefix.
\end{proof}

Finite rich words could always be extended richly with some letter by Proposition \ref{p6}.
The next proposition tells that the same holds for infinite rich words also.

\begin{proposition}
If $w$ is a right infinite rich word then there always exists a letter $a\in{\rm Alph}(w)$ such that $aw$ is rich.
\end{proposition}
\begin{proof}
We suppose to the contrary that $aw$ is not rich for every $a\in{\rm Alph}(w)$.
From the definition of a rich infinite word we get that for every $a\in{\rm Alph}(w)$
there exists a word $u_a$ such that $u_a$ is non-rich prefix of $aw$.
Suppose that $u_b$ is longest such word.
Now we would get that $b^{-1}u_b$ is a rich word and cannot be extended richly to the left with any letter.
This is a contradiction because of Proposition \ref{p6}.
\end{proof}

The next proposition tells that taking the palindromic closure of the reverse of a rich word,
i.e. the left palindromic closure, preserves the number of rich extensions that the original word had.

\begin{proposition}
If $w$ is rich and can be extended richly in $n$ ways, then so can $\bar{w}^{(+)}$.
\end{proposition}
\begin{proof}
Let $w$ be a rich word and $\bar{w}^{(+)}=uw=\bar{w}\bar{u}=uv\bar{u}$,
where $v$ is the lpp of $w$ and $u$ is possibly an empty word.
The word $u$ can be empty only if $w$ is a palindrome. The claim is clearly true
for palindromes so we can suppose $u\neq\epsilon$.
We suppose that $wa$ is rich and $wb$ is non-rich, for letters $a,b\in{\rm Alph}(w)$.
Now we only need to prove that 1) $\bar{w}^{(+)}a$ is rich and 2) $\bar{w}^{(+)}b$ is non-rich.

1) Suppose that $p$ is the lps of $wa$, which means it is unioccurent in $wa$.
We will prove that $p$ is also unioccurent lps of $\bar{w}^{(+)}a=uwa=uv\bar{u}a$.
Clearly $p$ cannot occur inside $w$ nor $\bar{w}$, otherwise it would not be unioccurent in $wa$.
So if $p$ would not be unioccurent in $\bar{w}^{(+)}a$ then it would have to contain $v$.

Suppose that $p=xvy$, where $xy\neq \epsilon$.
If $p=wa=v\bar{u}a$, then another occurence of $p$ in $\bar{w}^{(+)}a=uwa=\bar{w}\bar{u}a$ would imply another occurence of $v$ in $w$ ($u\neq\epsilon$).
If $p\neq wa=v\bar{u}a$, then the occurence of $p$ in the end of $wa$ would again directly imply another occurence of $v$ in $w$ ($xy\neq \epsilon$). 
Suppose that $xy=\epsilon$, i.e. $p=v$. This would directly mean that $wa=v\bar{u}a=p\bar{u}a$ would have two occurences of $p$.
So in every case we get a contradiction.

2) If the lps of $\bar{w}^{(+)}b$ is shorter than $wb$ then it is not unioccurent, because $wb$ was non-rich. This means $\bar{w}^{(+)}b$ is not rich.
The lps of $\bar{w}^{(+)}b$ clearly cannot be $wb$, otherwise $wb$ would be a palindrome and hence rich.
If the lps of $\bar{w}^{(+)}b$ is strictly longer than $wb=v\bar{u}b$ then $v$ would occur at least twice in $w$, which is impossible.
\end{proof}

At the end of this section we prove that every factor of any Sturmian word can always be extended richly in two ways.
Let us first define Sturmian words. In the following we suppose that all words are binary.

A word $w$ is \emph{balanced} if for every two factors $x,y\in \textrm{F}(w)$ of the same length
and for every letter $a\in\textrm{Alph}(w)$ the number $||x|_a-|y|_a|$ is at most $1$.
If a word is not balanced then it is \emph{unbalanced}.
An infinite word is \emph{Sturmian} if it is balanced and aperiodic.
A finite word is \emph{Sturmian} if it is a factor of an infinite Sturmian word.

First we need to know that all Sturmian words actually are rich in the first place.
The next proposition is from \cite{lgz} (Proposition 2). It was proved for trapezoidal words,
but Sturmian words are a subset of those.

\begin{proposition}\label{pp5}
If $w$ is Sturmian then it is rich.
\end{proposition}

The next two propositions are from \cite{bs} (Proposition 2.1.17 and Proposition 2.1.3, respectively).
The first one tells that finite Sturmian words and finite balanced words are the same and the second one
gives us a property for unbalanced words.

\begin{proposition}\label{pp6}
A finite word is Sturmian if and only if it is balanced.
\end{proposition}

\begin{proposition}\label{pp7}
A word $u$ is unbalanced if and only if there exists a palindrome $v$ such that $0v0$ and $1v1$ are factors of $u$.
\end{proposition}

The next proposition states that every finite Sturmian word,
and hence every finite balanced word, can be extended richly in two ways.

\begin{proposition}\label{pp8}
Every finite Sturmian word can always be extended richly in two ways.
\end{proposition}
\begin{proof}
Suppose $u$ is a finite Sturmian word and $\textrm{Alph}(u)=\{0,1\}$.
From Proposition \ref{pp5} and Proposition \ref{pp6} we get that $u$ is rich and balanced.
We only have to prove that both $u1$ and $u0$ are rich.

Suppose to the contrary that $u1$ is not rich (the case $u0$ is identical).
Now $u1$ is unbalanced because otherwise it would be Sturmian and hence rich.
Proposition \ref{pp7} tells us that there is a palindrome $v$ such that $0v0$ and $1v1$ are factors of $u1$.
Because $u$ was balanced, the factor $1v1$ has to be a suffix of $u1$ and it cannot occur anywhere else in the word.
This means that $1v1$ is a new palindrome and $u1$ is rich, a contradiction.
\end{proof}


\section{The infinite defect}

Rich words were defined such that there were maximum number of palindromes in the word.
We can define other words with respect to how many palindromes they lack, i.e. the defect of the word.
This concept have been studied in various papers from different angles, for example in \cite{gjwz}, \cite{bhnr}, \cite{br}, \cite{bps1} and \cite{bps2}.
In this section we define a new concept, the infinite defect.

The \emph{defect} of a finite word $w$ is defined by $\text{D}(w)=|w|+1-|\textrm{Pal}(w)|$.
The \emph{defect} of an (right, left or two-way) infinite word $w$ is defined by
$\text{D}(w)=\sup\{\text{D}(u) |\ u\ \text{is a factor of}\ w\}$.
If the previous supremum does not exists then the defect is defined to be $\infty$.
Clearly, finite and infinite rich words are exactly those words with defect equal to $0$.

We can also study how few defects a finite word must have if it has to be extended to be an infinite word.
We define the \emph{infinite defect} of a finite word $w$ with
$\text{D}_\infty(w)=\min\{\text{D}(z)\ |\ z\ \text{is an infinite word which has factor}\ w\}$,
where we of course suppose ${\rm Alph}(z)\subseteq{\rm Alph}(w)$.

The next theorem guarantees that the $\min$-function above gives always a finite number.

\begin{theorem}\label{tttt1}
The infinite defect ${\rm D}_\infty(w)$ of a finite word $w$ is finite.
\end{theorem}
\begin{proof}
Let us first denote $u=w^{(+)}$. We will prove that $u^\infty$ has finite defect.
More precisely, we will prove ${\rm D}(u^\infty)={\rm D}_\infty(u^2)$.
We do it with induction. So the claim is that ${\rm D}_\infty(u^n)={\rm D}_\infty(u^2)$ holds for all $n\geq2$.
If $n=2$ then the claim is trivial. Suppose it holds for $n=k$.

If $v$ is any non-empty prefix of $u$, then clearly $\bar{v}u^{k-1}v$ is a palindromic suffix of $u^k v$.
The word $\bar{v}u^{k-1}v$ is longer than half of the word $u^k v$,
so if $\bar{v}u^{k-1}v$ is not unioccurent in $u^k v$ then it must overlap with itself.
We can take the longest such overlap and get that it is unioccurent lps of $u^k v$.

We got that for every non-empty prefix $v$ of $u$ the word $u^k v$ has an unioccurent lps.
The word $u^{k+1}$ has therefore the same defect as $u^k$ because we get a new palindrome in every step when we extend $u^k$ into $u^{k+1}$.
This completes the induction.

We now get our claim because clearly ${\rm D}_\infty(w) \leq {\rm D}_\infty(u^2)$, where ${\rm D}_\infty(u^2)$ is finite.
\end{proof}

\begin{remark}
The definition of the infinite defect to be useful we have to find a word $w$ for
which ${\rm D}(w)\neq {\rm D}_\infty(w)$. The word $w=110100110111011001011$ is such.
Clearly ${\rm D}(w)=2$ but for every $w0, w1, 0w$ and $1w$ the defect is equal to $3$.
This means that no matter how we extend the word we always create new defects.

The defect and the infinite defect of rich words are equal to 0.
The defect and the infinite defect of a finite word can be same for non-rich words also.
For example ${\rm D}(00101100)={\rm D}_\infty(00101100)=1$.
\end{remark}

\begin{remark}
We could also define right and left infinite defects of $w$ separately such that the right (left) infinite
defect would mean the lowest defect of a right (resp. left) infinite word that contains $w$.
This would also be reasonable since there are words for which they would be different.
For example the word $w=101100111010111011$ has defect equal to $1$. We can extend $w$
to the right with infinite word $1^\infty$ and get a word which has also defect equal to $1$.
But already the words $0w$ and $1w$ have defects equal to $2$. So to get an infinite word that
has the lowest defect we sometimes have to extend it to the left and sometimes to the right.
Sometimes it does not matter.

We could also study how a right or left infinite word can be extended into a two-way infinite word.
The \emph{infinite defect} of a right or left infinite word $w$ could be defined by
${\rm D}_\infty(w)=\inf\{{\rm D}(z) |\ z\ \text{is a two-way infinite word that contains}\ w \}$.
\end{remark}

If a word is rich then we know that the infinite defect is always zero.
But how to determine the infinite defect for other words? We know that it is always finite.
The problem might be algorithmically undecidable, but we can at least find some upper bounds.
Clearly the normal defect is always a lower bound, which sometimes is achieved.

\begin{proposition}\label{pppp1}
For a finite word $w$ we have inequalities ${\rm D}_\infty(w)\leq{\rm D}(w^{(+)}w^{(+)})$ and 
${\rm D}_\infty(w)\leq{\rm D}(\bar{w}^{(+)}\bar{w}^{(+)})$.
\end{proposition}
\begin{proof}
This comes directly from the proof of Theorem \ref{tttt1}.
\end{proof}

\begin{remark}
This method is sometimes optimal. For example the word $w=00101100$ has defect equal to $1$
and so does the word $w^{(+)}w^{(+)}=0010110011010000101100110100$.

The method is still not always optimal. For example the word $w=110010010110010$ has defect equal to $4$.
We get that already the words $w^{(+)}=1100100101100100\cdots$ and $\bar{w}^{(+)}=0100110100100110\cdots$
have defects equal to $5$. But we can do better: ${\rm D}(1^\infty w)=4$, which means that ${\rm D}_\infty(w)=4$.

Note that the words $w^{(+)}w^{(+)}$ and $\bar{w}^{(+)}\bar{w}^{(+)}$ can actually have different defects.
Such word is for example $w=0010110001010$.
\end{remark}

\begin{proposition}\label{pppp2}
If $w$ is a finite word then ${\rm D}_\infty(w)\leq |w|-n$, where $n$ is the length of the 
longest rich suffix or the longest rich prefix of $w$.
\end{proposition}
\begin{proof}
We can take the longest rich suffix (prefix) $u$ of $w=vu$ and extend it to be a right (resp. left) infinite rich word.
When we add the left-over $v$ to the beginning (resp. to the end) we clearly get at most $|v|$ defects.
\end{proof}

\begin{remark}
The method described in the previous proposition is sometimes optimal.
For example, the word $w=00101100101$ has defect $3$
and so does the word $00101100101(0^\infty)$, where we have extended the longest rich suffix $01100101$.
We see that $|w|-n=11-8=3$.

Sometimes it is not optimal. The word $w=1101100111010011011001101101110011011$
has defect $16$. The longest rich suffix and prefix are $v=01110011011$ and $\bar{v}$, respectively.
If we want to extend those to be right and left infinite rich words, respectively, we have to extend
them both first with $0$. That creates a new defect. After that, both extensions $0$ and $1$ create new defects also.
So we have that ${\rm D}_\infty(w00)$, ${\rm D}_\infty(w01)$, ${\rm D}_\infty(00w)$, ${\rm D}_\infty(10w)\geq 18$.
But we can see that ${\rm D}_\infty(w)=17$. Consider the word $w1^\infty$: the first letter $1$ creates a new defect
(which we cannot avoid) but after that we always get a new palindrome $1^k$, where $k\geq4$.
\end{remark}

\begin{proposition}\label{pppp3}
If $w$ is a finite word then ${\rm D}_\infty(w)\leq{\rm D}(w)+n$, where $n$ is
the length of the biggest power of any letter in $w$. If $w$ ends or begins with $a^k$ then we can
choose $n=n_a - k$, where $a^{n_a}$ is the biggest power of letter $a$ in $w$.
\end{proposition}
\begin{proof}
Suppose $a^n$ is the biggest power of letter $a$ in $w$. If we now extend the word $w$ with $a^\infty$
then we create a new palindrome $a^m$, where $m > n$, in every step after we have extended $w$ with at least $a^n$.
So we create at most $n$ new defects. Clearly, if $w$ ends or begins with $a^k$ then we create at most $n-k$ new defects.
\end{proof}

\begin{remark}
This method is also sometimes optimal. For example if we look the previous remark where the defect was $16$.
Now the biggest power of letter $1$ is $1^3$ and the word $w$ ends with $1^2$. So we get that 
${\rm D}_\infty(w)\leq{\rm D}(w)+3-2=17$.

For the word $w=101001111000111101001$, which has defect $4$, this method is not optimal.
The biggest powers of letters in $w$ are $0^3$ and $1^4$. If we now extend the word with $0^\infty$
or $1^\infty$ to the left or right we get at least one defect more in each case.
But we can see that ${\rm D}(w(01)^\infty)=4$, which means ${\rm D}_\infty(w)=4$.
\end{remark}

Note that the propositions that give us those upper bounds give also a method
to construct the infinite word which contains the original word $w$.
The actual constructions in Proposition \ref{pppp2} and Proposition \ref{pppp3}
may sometimes give a smaller defect than what the bound is always guaranteeing.

Note also that, as the next proposition says, we can extend a finite word to be both periodic and aperiodic infinite word
such that the defect of that infinite word is still finite.

\begin{proposition}
If $w$ is a finite word then we can extend it to be both periodic and aperiodic infinite word $v$
such that ${\rm D}(v)$ is finite.
\end{proposition}
\begin{proof}
The periodic infinite word $v$ can be constructed using Theorem \ref{tttt1}: $v=(w^{(+)})^\infty$.
The aperiodic infinite word $v$ can be constructed using Proposition \ref{pppp2} and Proposition \ref{pp3}:
we take the longest rich suffix and extend it to be aperiodic rich word.
\end{proof}

Next, let us look how the defect and the infinite defect can differ from each other. 
First we define some functions for an integer $n$:
\begin{align*}
{\rm D}(n) & = \max\{ {\rm D}(w)\ |\ w\ \text{is a word of length at most}\ n\}, \\
{\rm D}_\infty(n) & = \max\{ {\rm D}_\infty(w)\ |\ w\ \text{is a word of length at most}\ n\}, \\
{\rm D_{dif}}(n) & = \max\{ {\rm D}_\infty(w)-{\rm D}(w)\ |\ w\ \text{is a word of length at most}\ n\}.
\end{align*}
We clearly see that ${\rm D}(n)$ and ${\rm D}_\infty(n)$ are non-bounded growing functions.
For the function ${\rm D_{dif}}(n)$ we can prove the following inequality.

\begin{proposition}
${\rm D}_\infty(n)-{\rm D}(n) \leq {\rm D_{dif}}(n)$.
\end{proposition}
\begin{proof}
Let us choose $w$ to be a word for which ${\rm D}_\infty(w)={\rm D}_\infty(n)$.
Now clearly ${\rm D}(w)\leq {\rm D}(n)$.
From these we get that ${\rm D}_\infty(n)-{\rm D}(n) \leq {\rm D}_\infty(w)-{\rm D}(w)\leq {\rm D_{dif}}(n)$.
\end{proof}

The previous proposition still does not guarantee that ${\rm D_{dif}}(n)$ would be non-bounded because
we do not know how much faster ${\rm D}_\infty(n)$ grows than ${\rm D}(n)$, if any.
But with a quite artificial construction we can prove it.

\begin{proposition}
The function ${\rm D_{dif}}(n)$ is a non-bounded growing function.
\end{proposition}
\begin{proof}
It is trivially growing, so we only need to prove the non-boundedness.

Suppose to the contrary that there exists $k>1$ such that $\forall n: {\rm D_{dif}}(n) < k$ .
We will consider the word $u=001^{k+1}01w101^{k+1}00$, where $w$ is
a word which does not have $1^{k+1}$ as a factor but still has every
palindrome of length at most $2k+2$ as a factor that is possible under this restriction.
Now we need to prove that no matter how we extend $w$ to be an infinite word, the defect will always grow with at least $k$.

We cannot get a lower defect by extending the word to the both ends, so we suppose we extend it only to the right.
Every time we put a letter to the end of $u$ we cannot get a new palindromic factor that would contain $1^{k+1}$.
This comes from the fact that $w$ does not contain $1^{k+1}$ and the factor $1^{k+1}$ is preceded by $10$ and followed by $00$.
This means that after we have extended the word $u$ with $k$ letters we have not got any new palindromes
because every palindrome of length $2k+2$ or shorter is contained in $w$.
\end{proof}


\section{The number of rich words}

All binary words of length $7$ or shorter are rich. The shortest non-rich binary words are of length $8$
and there are four of them: $00101100$, $00110100$, $11010011$, $11001011$. Since all words are not rich
it is natural to study how many of them exist. This has not yet been done.

Let us mark with $r_k(n)$ the number of $k$-ary rich words of length $n$.
The next proposition states that there is an exponentially decreasing upper bound for $r_2(n)/2^n$
(i.e. the percentage of rich binary words from all binary words).

\begin{proposition}\label{ppp1}
The function $r_2(n)/2^n$ has an exponentially decreasing upper bound. 
\end{proposition}
\begin{proof}
First we give a recursive upper bound for $r_2(n)$.
We start with the exact initial values: $r_2(k)=2^k$ for $0\leq k\leq 7$.
Because there are four words of length $8$ that are not rich and every binary rich word can be extended richly at most two ways,
we get a recursive upper bound $r_2(n)\leq  2^8 r_2 (n-8)- 4 r_2 (n-8)=252 r_2 (n-8)$ for $n\geq 8$.
This kind of recursive function is easy to solve: $r_2(n)\leq 252^{\lfloor n/8 \rfloor}2^{n-8\lfloor n/8 \rfloor}$

Now we have that $r_2(n)/2^n \leq 252^{\lfloor n/8 \rfloor}2^{n-8\lfloor n/8 \rfloor} / 2^n
= 252^{\lfloor n/8 \rfloor}2^{-8\lfloor n/8 \rfloor}=(63/64)^{\lfloor n/8 \rfloor}$.
The function $(63/64)^{\lfloor n/8 \rfloor}$ decreases to $0$ exponentially when $n$ goes to infinity.
\end{proof}

Using the idea from the previous proof, we can trivially enhance the upper bound.
We just need to note that there are $16$ non-rich words of length $9$, $44$ non-rich words of length $10$,
$108$ non-rich words of length $11$ and $266$ non-rich words of length $12$, such that all the non-rich words
do not contain the shorter ones as a suffix (so they really do create completely new non-rich words).

\begin{corollary}\label{ccc1}
For $n\geq 12$ we have $r_2(n)\leq 2 r_2(n-1) - 4 r_2(n-8) - 16 r_2(n-9) - 44 r_2(n-10) - 108 r_2(n-11) - 266 r_2(n-12)$,
where we have the exact values of $r_2(k)$ for $0\leq k\leq 11$.
\end{corollary}

We can of course do the same for every size of the alphabet if we change few numbers from the upper proof. Suppose $k\geq3$.
All $k$-ary words of length $3$ or shorter are rich but all the words of form $0120$ are non-rich and
there are $k(k-1)(k-2)$ of them.

\begin{corollary}\label{ccc2}
The funtion $r_k(n)/k^n$ has an exponentially decreasing upper bound $(1-k(k-1)(k-2)/k^4)^{\lfloor n/4 \rfloor}$.
\end{corollary}

For the lower bound of $r_k(n)$ we use Theorem \ref{tt1}. We know that every rich word can be extended richly at least one
way and from Theorem \ref{tt1} we get that after an extension of limited length we can extend it richly with two letters. 

\begin{proposition}\label{ppp2}
For $n\geq1$ we have $r_k(n)\geq r_k(n-1)+r_k(\lfloor n/3\rfloor)$, where $r_k(0)=1$.
\end{proposition}
\begin{proof}
Every rich word can extended richly at least in one way.
From Theorem \ref{tt1} we get that every rich word $w$ of length $\lfloor n/3\rfloor$ can be extended richly with at least two ways after it has
been extended with a proper word $u$ such that $|wu|=n$. These facts clearly give us our recursive formula when we notice that $r_k(0)=1$.
\end{proof}

\begin{figure}
\centering
\includegraphics[width=0.9\textwidth]{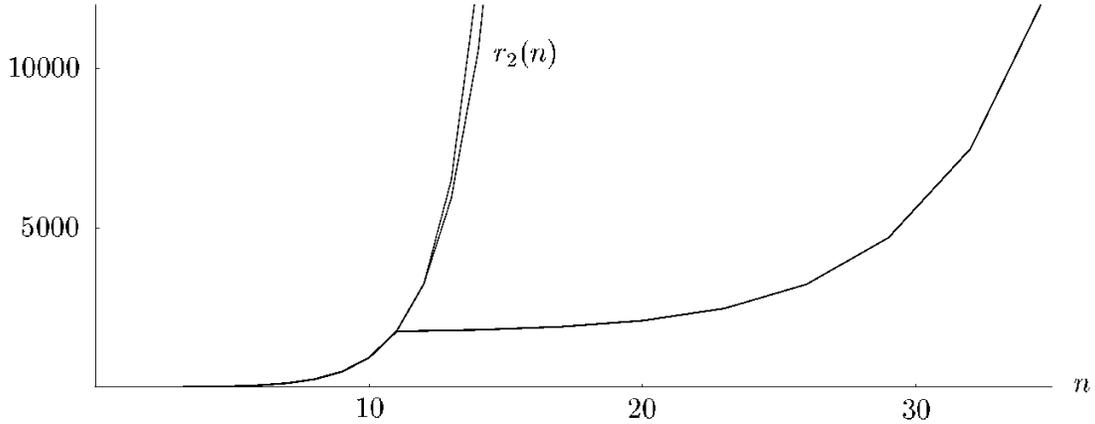}
\caption{The number of rich binary words and the bounds.}
\label{pic1}
\end{figure}

\begin{remark}
We do not know whether the function that comes from the
recursive formula of the lower bound is exponentially growing or not.
Note that we can improve the lower bound if we can improve Theorem \ref{tt1}.
Note also that Remark \ref{rr1} gives some limitations on improving it.
\end{remark}

For $0\leq n\leq 25$ the exact numbers of $r_2(n)$ are: 1, 2, 4, 8, 16, 32, 64, 128, 252, 488, 932, 1756,
3246, 5916, 10618, 18800, 32846, 56704, 96702, 163184, 272460, 450586, 738274, 1199376, 1932338, 3089518.
In Figure \ref{pic1} there are both the bounds and the 26 first exact numbers of rich binary words.
The picture would imply that the upper bound cannot be enhanced very much but the lower bound can be.
This means that there is much to improve in Theorem \ref{tt1}.


\section{Two-dimensional rich words}

Let $A$ be an alphabet. We define \emph{two-dimensional word} to be an infinite rectangular grid $\mathbb{Z}^2$
where every pair $(i,j)\in\mathbb{Z}^2$ gets a value from $A\cup\{\epsilon\}$.
We denote with $w(i,j)$ the letter (or the empty word)
from the pair of indices $(i,j)\in\mathbb{Z}^2$ of a two-dimensional word $w$.
Two-dimensional words have been studied in various papers, for example in \cite{bv}, \cite{ekm}, \cite{qz} and \cite{pa},
but these concepts have not been applied to rich words.
Here we define rich two-dimensional words and make few notions about them.

Two-dimensional word $w$ is \emph{rich} if for every $i,j\in\mathbb{Z}$ all the words
$w(i,j)$ $w(i,j+1)\cdots w(i,j+n)$ and $w(i,j)w(i+1,j)\cdots w(i+n,j)$ are rich,
where $n\geq0$ and $w(i,j+k),w(i+k,j)\neq\epsilon$ for every $k$ ($0\leq k\leq n$).

We say that a two-dimensional word $w$ can be extended to be a \emph{rich plane} if for every $(i,j)\in\mathbb{Z}^2$
for which $w(i,j)=\epsilon$ there exists a letter $a\in\textrm{Alph}(w)$ such that if we set $w(i,j)=a$
then the new two-dimensional word is rich.

\begin{example}
The two-dimensional words in Figure \ref{pic2} are both rich. The previous can trivially
be extended to be a rich plane by just adding letter $1$ for every empty spot, but the latter cannot.
This comes from the fact that we would be forced to extend the $8\times8$--square to the right so that the vertical word
becomes $00101100$, which is non-rich.
\end{example}

\begin{figure}[htb!]
\centering
\includegraphics[width=0.8\textwidth]{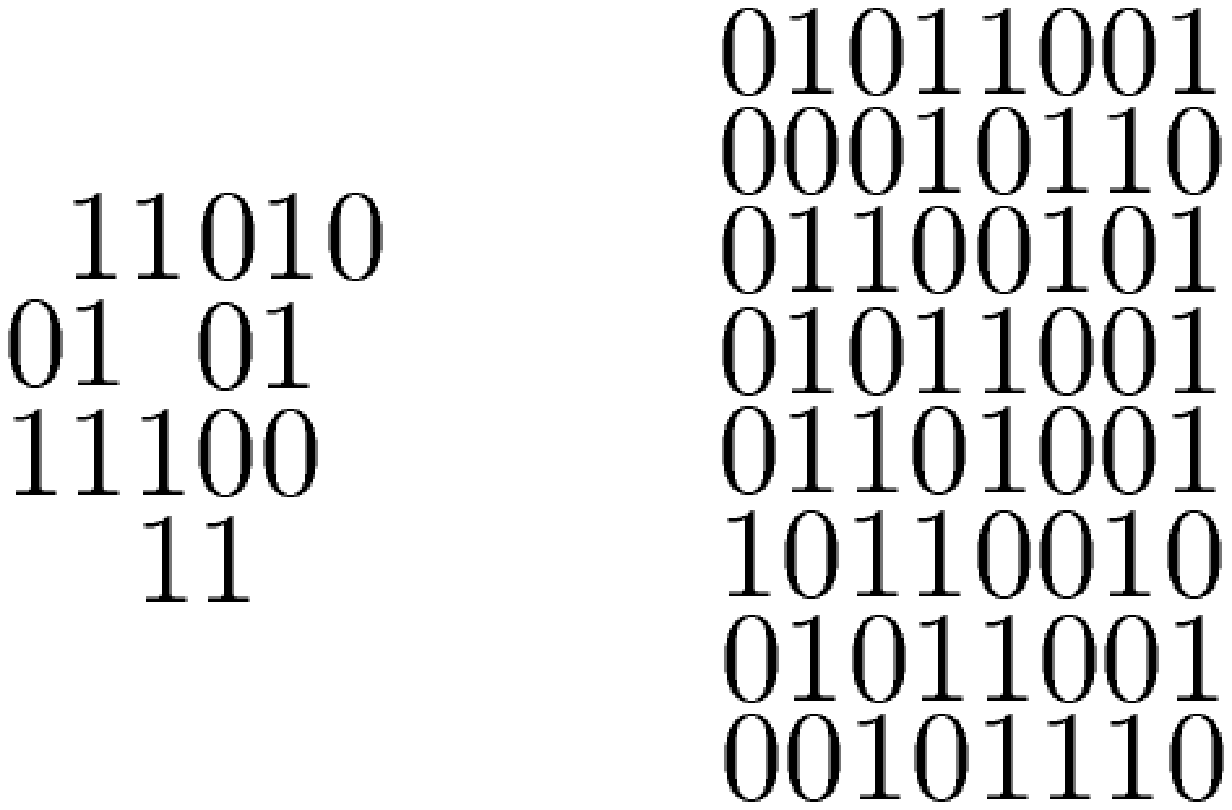}
\caption{Two rich two-dimensional words.}
\label{pic2}
\end{figure}

Let us suppose we have binary alphabet $\{0,1\}$. We saw that $8\times8$--squares cannot always be extended to be rich planes.
The next proposition states that every $6\times6$--square can always be extended to be a rich plane.
The $6\times6$--square does not need to be full with letters because every binary word of length $6$ or shorter is rich.

\begin{proposition}
If $w$ is a two-dimensional binary word such that every $w(i,j)\neq\epsilon$ is inside a $6\times6$--square,
then $w$ can be extended to be a rich plane.
\end{proposition}
\begin{proof}
The proof is constructional. We suppose that the corners of the $6\times6$--square of $w$ are $(1,1,1,1)$,
$(1,1,1,0)$, $(0,0,1,1)$ or $(0,1,0,1)$,
where the order of the corners is left lower, left upper, right upper and right lower.
Clearly all other possibilities are isomorphic in terms of the orientation of the plane and/or swapping the letters.
If the $6\times6$--square is not full, we can extend it to be full in any way we want.

We extend all the four possibilities in a way that is described in Figure \ref{pic3}.
The big letters $0$ and $1$ mean that the whole part of the plane is filled with that letter.
All the horizontal
\begin{figure}[htb!]
\centering
\includegraphics[width=0.6\textwidth]{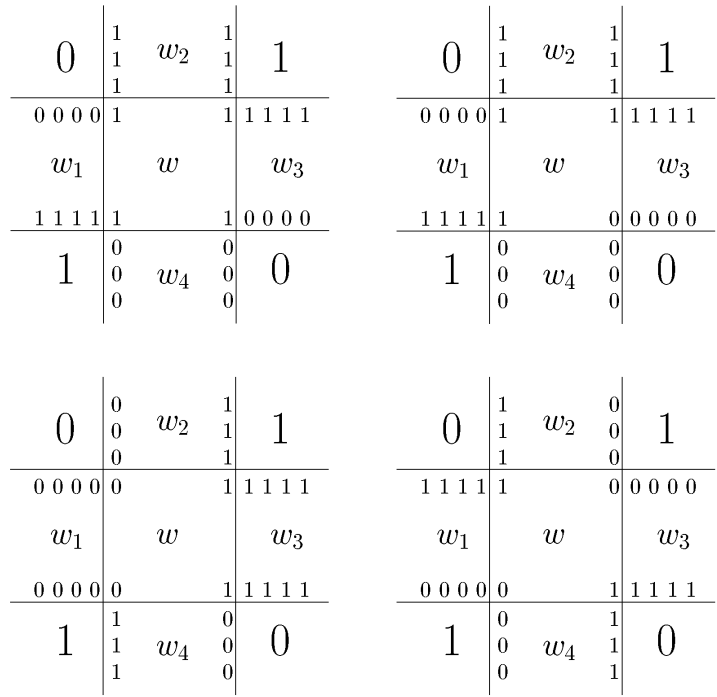}
\caption{The rich extensions of all $6\times6$--squares.}
\label{pic3}
\end{figure}
and vertical words inside the $6\times6$--square can be extened to be infinite rich words.
The words $w_1,w_2,w_3$ and $w_4$ are any words that satisfies this.

Now we can see that the whole plane is rich by using the fact that the words $0^\infty u 1^\infty$,
$0^\infty u 0 1^\infty$ and $0^\infty 1 u 1^\infty$ are always rich if $|u|\leq 4$.
\end{proof}

The next trivial proposition gives some other two-dimensional words that can be extended to be rich planes.
A finite word $w$ is \emph{strongly} rich if $w^\infty$ is rich. For strongly rich words, look \cite{gjwz} and \cite{rr}.

\begin{proposition}
Let $w$ be a two-dimensional word for which $w(i,j)\neq\epsilon$ only inside some rectangle.
If every $w(i,j)=\epsilon$ inside that rectangle can be chosen to be some letter such that every horizontal and vertical word is strongly rich,
then $w$ can be extended to be a rich plane.
\end{proposition}
\begin{proof}
We simply extend $w$ to be strongly rich inside that rectangle and then we just replicate
it and fill the whole plane so that the rectangles are side by side. Clearly the two-dimensional word that we get is rich.
\end{proof}

The problem whether a given two-dimensional word $w$ can be extended to be a rich plane is of course semi-decidable,
i.e. there exists an algorithm that gives "yes" if the word cannot be extended.
It just tries every possible way to fill the empty indices $(i,j)\in\mathbb{Z}^2$ of $w$ in some order.
If at some point there is no possible choice, the algorithm halts and returns "yes".
The problem is although probably not decidable.


\section{Open problems}

Here we list some open problems from the previous sections.

\begin{openprob}
Let $w$ be a rich word. How long is the shortest $u$ such that $wu$ can always be extended at least in two ways?
\end{openprob}

\begin{openprob}
Is the condition in Proposition \ref{pj} sufficient for joining two rich words $u$ and $v$ into factors of a same rich word?
\end{openprob}

\begin{openprob}
Is the following problem decidable: "Determine the infinite defect of a given finite word $w$"?
\end{openprob}

\begin{openprob}
Is the funtion $r_k(n)$, i.e. the number of rich words, exponentially increasing?
\end{openprob}

\begin{openprob}
Can every $7\times 7$--square be extended to be a rich plane?
\end{openprob}

\begin{openprob}
Is the following problem decidable: "Given a two-dimensional word $w$, can it be extended to be a rich plane"?
\end{openprob}


\section*{Acknowledgements}

I would like to thank my supervisors Luca Zamboni and Tero Harju for several fruitful discussions and for introducing me to rich words.

%
%
%
%
%
%
%





\begin{thebibliography}{NHCQ}

\small

\bibitem[AFMP]{afmp}
\textsc{P. Ambro\v z, C. Frougny, Z. Mas\'akov\'a, E. Pelantov\'a}:
\emph{Palindromic complexity of infinite words associated with simple Parry numbers},
Ann. Inst. Fourier (Grenoble) 56 (2006) 2131-2160.

\bibitem[BPS1]{bps1}
\textsc{L. Balkov\'a, E. Pelantov\'a, \v S. Starosta}:
\emph{Infinite words with finite defect},
Advances in Applied Mathematics 47 (2011) 562-574.

\bibitem[BPS2]{bps2}
\textsc{L. Balkov\'a, E. Pelantov\'a, \v S. Starosta}:
\emph{Proof of the Brlek-Reutenauer conjecture},
Theoret. Comput. Sci. 475 (2013) 120-125.

\bibitem[BS]{bs}
\textsc{J. Berstel, P. S\'e\'ebold}:
\emph{Sturmian words,}
in: M. Lothaire (Ed.), Algebraic Combinatorics on Words,
Cambridge University Press, Cambridge, UK, 2002.

\bibitem[BV]{bv}
\textsc{V. Berthe, L. Vuillon}:
\emph{Tilings and rotations: A two-dimensional generalization of Sturmian sequences},
Discrete Math. 223 (2000) 27-53.

\bibitem[BHNR]{bhnr}
\textsc{S. Brlek, S. Hamel, M. Nivat, C. Reutenauer}:
\emph{On the palindromic complexity of infinite words},
Internat. J. Found. Comput. 15 (2004) 293-306.

\bibitem[BR]{br}
\textsc{S. Brlek, C. Reutenauer}:
\emph{Complexity and palindromic defect of infinite words},
Theoret. Comput. Sci. 412 (2011) 493-497.

\bibitem[BDGZ1]{bdgz1}
\textsc{M. Bucci, A. De Luca, A. Glen, L. Q. Zamboni}:
\emph{A connection between palindromic and factor complexity using return words},
Advances in Applied Mathematics 42 (2009) 60-74.

\bibitem[BDGZ2]{bdgz}
\textsc{M. Bucci, A. De Luca, A. Glen, L. Q. Zamboni}:
\emph{A new characteristic property of rich words},
Theoret. Comput. Sci. 410 (2009) 2860-2863.

\bibitem[DGZ]{lgz}
\textsc{A. de Luca, A. Glen, L. Q. Zamboni}:
\emph{Rich, Sturmian, and trapezoidal words},
Theoret. Comput. Sci. 407 (2008) 569-573.

\bibitem[DJP]{djp}
\textsc{X. Droubay, J. Justin, G. Pirillo}:
\emph{Episturmian words and some constructions of de Luca and Rauzy},
Theoret. Comput. Sci. 255 (2001) 539-553.

\bibitem[EKM]{ekm}
\textsc{C. Epifanio, M. Koskas, F. Mignosi}:
\emph{On a conjecture on bidimensional words},
Theoret. Comput. Sci. 299 (1-3) (2003) 123-150.

\bibitem[GJWZ]{gjwz}
\textsc{A. Glen, J. Justin, S. Widmer, L. Q. Zamboni}:
\emph{Palindromic richness},
European Journal of Combinatorics 30 (2009) 510-531.

\bibitem[Lot1]{l1}
\textsc{M. Lothaire}:
\emph{Combinatorics on Words,}
in: Encyclopedia of Mathematics and its Applications, vol 17,
Addison-Wesley, Reading, MA, 1983.

\bibitem[Lot2]{l2}
\textsc{M. Lothaire}:
\emph{Algebraic Combinatorics on Words,}
in: Encyclopedia of Mathematics and its Applications, vol 90,
Cambridge University Press, UK, 2002.

\bibitem[PA]{pa}
\textsc{S. A. Puzynina, S. V. Avgustinovich}:
\emph{On periodicity of two-dimensional words},
Theoret. Comput. Sci. 391 (2008) 178-187.

\bibitem[QZ]{qz}
\textsc{A. Quas, L. Q. Zamboni}:
\emph{Periodicity and local complexity},
Theoret. Comput. Sci. 319 (1-3) (2004) 229-240.

\bibitem[RR]{rr}
\textsc{A. Restivo, G. Rosone}:
\emph{Burrows-Wheeler transform and palindromic richness},
Theoret. Comput. Sci. 410 (2009) 3018-3026.


\end{thebibliography}
\end{document}